\documentclass[11pt, british]{amsart}

\usepackage{babel,eucal,url,amssymb,
enumerate,amscd,
}
\usepackage{color}

 \newtheorem{thm}{Theorem}[section]
 \newtheorem{cor}[thm]{Corollary}
 
 \newtheorem{prop}[thm]{Proposition}
 \newtheorem{defn}[thm]{Definition}
 \newtheorem{rem}[thm]{Remark}
 
 \numberwithin{equation}{section}

\begin{document}

\title[]
 {Null $\varphi $--Slant Curves in a Main Class of\\3-Dimensional Normal Almost 
 Contact \\ B-Metric Manifolds}

\author[G. Nakova]{Galia Nakova}
\address{
University of Veliko Tarnovo "St. Cyril and St. Methodius" \\ Faculty of Mathematics and Informatics
\\ 2 Teodosii Tarnovski Str., Veliko Tarnovo 5003, Bulgaria}
\email{gnakova@gmail.com}




\begin{abstract}
We introduce a new type of slant curves in almost contact B-metric manifolds, called 
$\varphi $-slant curves,  by an additional condition which is specific for these manifolds.
In this paper we study $\varphi $-slant null curves in a class of 3-dimensional normal almost contact B-metric manifolds and prove that for non-geodesic of them there exists a unique Frenet frame for which the original parameter is distinguished. We investigate some of $\varphi $-slant null  curves and with respect to the associated B-metric on the manifold and find relationships between the corresponding Frenet frames and curvatures. We construct the examined curves  in a 3-dimensional Lie group  and give their matrix representation.
\end{abstract}

\newcommand{\ie}{i.\,e. }
\newcommand{\g}{\mathfrak{g}}
\newcommand{\D}{\mathcal{D}}
\newcommand{\F}{\mathcal{F}}
\newcommand{\diag}{\mathrm{diag}}
\newcommand{\End}{\mathrm{End}}
\newcommand{\im}{\mathrm{Im}}
\newcommand{\id}{\mathrm{id}}
\newcommand{\Hom}{\mathrm{Hom}}

\newcommand{\Rad}{\mathrm{Rad}}
\newcommand{\rank}{\mathrm{rank}}
\newcommand{\const}{\mathrm{const}}
\newcommand{\tr}{{\rm tr}}
\newcommand{\ltr}{\mathrm{ltr}}
\newcommand{\codim}{\mathrm{codim}}
\newcommand{\Ker}{\mathrm{Ker}}
\newcommand{\R}{\mathbb{R}}
\newcommand{\K}{\mathbb{K}}

\newcommand{\thmref}[1]{Theorem~\ref{#1}}
\newcommand{\propref}[1]{Proposition~\ref{#1}}
\newcommand{\corref}[1]{Corollary~\ref{#1}}
\newcommand{\secref}[1]{\S\ref{#1}}
\newcommand{\lemref}[1]{Lemma~\ref{#1}}
\newcommand{\dfnref}[1]{Definition~\ref{#1}}


\newcommand{\ee}{\end{equation}}
\newcommand{\be}[1]{\begin{equation}\label{#1}}

\maketitle

\section{Introduction}\label{sec-1}
In the Lorentzian geometry there exist three types of curves according to the causal character of their tangent vector -- spacelike, timelike and null (lightlike) curves. Studying the geometry of null curves is of special interest since they have very different properties compared to spacelike and timelike curves. 
The general theory of null curves is developed in \cite{D-B, D-J},
where there are established important applications of these curves in general relativity. 

Let $\bf{F}$ be a Frenet frame along a null curve $C$ on a Lorentzian manifold. According to \cite{D-B}, $\bf{F}$ and the Frenet equations with respect to $\bf{F}$ depend on both the parametrization of $C$ and the choice of a screen vector bundle. However, if a non-geodesic
null curve $C$ is  properly parameterized, then there exists only one Frenet frame, called a Cartan Frenet frame, for which
the corresponding Frenet equations of $C$, called Cartan Frenet equations, have minimum number of curvature functions (\cite{D-J}).

In this paper we consider 3-dimensional almost contact B-metric manifolds $(M,\varphi,\xi,\eta,g)$, which are Lorentzian manifolds equipped with an almost contact B-metric structure. We study $\varphi $-slant null curves in considered manifolds belonging to the class $\F_4$ of the Ganchev-Mihova-Gribachev classification given in \cite{GaMGri}. The class $\F_4$ consists of normal almost contact B-metric manifolds and it is analogous to the class of $\alpha $-Sasakian manifolds in the theory of almost contact metric manifolds.

A slant curve $C(t)$ on $(M,\varphi,\xi,\eta,g)$, defined by the
condition $g(\dot C,\xi )=a={\rm const}$ for the tangent vector $\dot C$, is a natural generalization of a cylindrical helix in an Euclidean space. Slant curves and in particular Legendre curves (which are slant curves with $a=0$) in almost contact metric and almost paracontact metric manifolds have been investigated  intensively by many  authors \cite{I-L, W} and the references therein. 

In the present work we introduce a new type of slant curves in almost contact B-metric manifolds, called 
$\varphi $-slant curves,  by the additional condition $g(\dot C,\varphi \dot C)=b={\rm const}$. For these manifolds, in contrast to the almost contact metric and almost paracontact metric manifolds, $b$ is  a non-zero function in general. 

The paper is organized as follows.
Section 2 is a brief review of almost contact B-metric manifolds and geometry of null curves in a 3-dimensional Lorentzian manifold.
First in Section 3 we show that in a 3-dimensional almost contact B-metric manifold there exist no 
$\varphi $-slant null curves such that $a=b=0$. Then we
prove that for a non-geodesic $\varphi $-slant null curve $C(t)$ in a 3-dimensional 
$\F_4$-manifold there exists a unique Frenet frame ${\bf F}_1$ for which the original parameter $t$ is distinguished, as well as we express ${\bf F}_1$ in terms of the almost contact B-metric structure. Also, we find the curvatures $k_1(t)$ and $k_2(t)$ with respect to ${\bf F}_1$. On an almost contact B-metric manifold there exist two B-metrics $g$ and $\widetilde g$. For that reason in Section 4 we consider some
$\varphi $-slant null curves with respect to $g$  in a 3-dimensional $\F_4$-manifold and prove that with respect to $\widetilde g$ these curves are $\varphi $-slant non-null curves. Moreover, we obtain  relationships between the Frenet frames and the corresponding curvatures with respect to  $g$ and 
$\widetilde g$.
In the last Section 4 we construct $\varphi $-slant null curves in a 3-dimensional Lie group endowed with an almost contact B-metric structure of an $\F_4$-manifold. We find a matrix representation of considered curves.

\section{Preliminaries}\label{sec-2}
A $(2n+1)$-dimensional smooth manifold  $(M,\varphi,\xi ,\eta ,g)$ is called an almost contact manifold with B-metric  (or {\it an almost contact B-metric manifold}) \cite{GaMGri} if it is endowed with an almost contact structure $(\varphi ,\xi ,\eta )$ consisting of an endomorphism $\varphi $ of the tangent bundle, a Reeb vector field $\xi $  and its dual 1-form $\eta $, satisfying the following relations:
\begin{align*}
\varphi^2X=-X+\eta(X)\xi, \qquad \quad \eta(\xi )=1.
\end{align*}
Also, $M$ is equipped with a semi-Riemannian metric $g$, called {\it a B-metric} \cite{GaMGri}, determined by
\[
g(\varphi X,\varphi Y)=-g(X,Y)+\eta (X)\eta (Y).
\]
Here and further $X$, $Y$, $Z$ are tangent vector fields on $M$, \ie $X,Y,Z \in TM$.
Immediate consequences of the above conditions are:
\begin{align*}
\eta \circ \varphi =0, \quad \varphi \xi =0, \quad {\rm rank}(\varphi)=2n, \quad \eta (X)=
g(X,\xi ), \quad g(\xi,\xi )=1.
\end{align*}
The distribution $\mathbb {D}: x \in M \longrightarrow \mathbb {D}_x\subset T_xM$, where
\[
\mathbb D_x=Ker \eta=\{X_x\in T_xM: \eta (X_x)=0\}
\]
is called {\it a contact distribution} generated by 
$\eta $. Then the tangent space $T_xM$ at each $x\in M$ is the following orthogonal direct sum
\[
T_xM=\mathbb D_x\oplus span_\mathbb R\{\xi _x\} .
\]
The tensor field $\varphi $ induces an almost complex structure on each
fibre on $\mathbb D$. Since $g$ is non-degenerate metric on $M$ and $\xi $ is non-isotropic,
the contact distribution $\mathbb D$ is non-degenerate and the restriction $g_{\vert \mathbb D}$ of the metric $g$ on $\mathbb D$ is of signature $(n,n)$.
\\
The tensor field ${\widetilde g}$ of type $(0,2)$ given by 
${\widetilde g}(X,Y)=g(X,\varphi Y)+\eta (X)\eta (Y)$
is a B-metric, called {\it an associated metric} to $g$. Both metrics $g$ and 
${\widetilde g}$ are necessarily of signature $(n+1,n)$ 
$(+\ldots + -\ldots -)$.
\\
Let $\nabla$ be the Levi-Civita connection of $g$. 
The tensor field $F$ of type $(0,3)$ on $M$ is defined by
$F(X,Y,Z)=g((\nabla_X\varphi)Y,Z)$ 
and it has the following properties:
\[
F(X,Y,Z)=F(X,Z,Y)=F(X,\varphi Y,\varphi Z)+\eta (Y)F(X,\xi,Z)+\eta (Z)F(X,Y,\xi ).
\]
Moreover, we have
\begin{align}\label{2.1}
F(X,\varphi Y,\xi )=(\nabla _X\eta )Y=g(\nabla _X\xi,Y).
\end{align}
The following 1-forms, called \emph{Lee forms}, are associated with $F$:
\[
\theta (X)=g^{ij}F(e_i,e_j,X), \quad \theta ^*(X)=g^{ij}F(e_i,\varphi e_j,X), \quad
\omega (X)=F(\xi,\xi,X),
\]
where $\{e_i\}, \, i=\{1,\ldots,2n+1\}$ is a basis of $T_xM$, $x\in M$, and $(g^{ij})$ is the inverse matrix of $(g_{ij})$.
\\
A classification of the almost contact B-metric manifolds with respect to $F$ is given in \cite{GaMGri} and
eleven basic classes  $\F_i$ $(i=1,2,\dots,11)$ are obtained. If $(M,\varphi,\xi,\eta,g)$ belongs to the class $\F_i$ then it is called an \emph{$\F_i$-manifold}.
\\
The special class $\F_0$ is the intersection of all basic classes. It is known as the class of the {\it cosymplectic B-metric manifolds}, \ie the class of the considered manifolds with parallel structure tensors with respect to $\nabla$, namely
$\nabla \varphi =\nabla \xi =\nabla \eta =\nabla g=\nabla {\widetilde g}=0$ and consequently $F=0$.
\\
The lowest possible dimension of the considered manifolds is three. The class of the 3-dimensional almost contact B-metric manifolds is
$\F_1\oplus \F_4\oplus \F_5\oplus \F_8\oplus \F_9\oplus \F_{10}\oplus \F_{11}$ \cite{GaMGri}. According to \cite{ManIv13}, the class of the normal almost contact B-metric manifolds is $\F_1\oplus\F_2\oplus\F_4\oplus\F_5\oplus\F_6$, since the Nijenhuis tensor of almost contact structure vanishes there. Hence, the class of the 3-dimensional normal almost contact B-metric manifolds is $\F_1\oplus\F_4\oplus\F_5$. 
\\
Let us remark that only in the classes $\F_1$, $\F_4$, $\F_5$ and $\F_{11}$ the structure tensor $F$ is expressed explicitly by the 1-forms $\theta $, $\theta ^*$, $\omega $ and the basic tensors of type $(0,2)$ $g$, $\widetilde g$, $\eta \otimes \eta $ of the manifold. In this reason, the classes $\F_1$, $\F_4$, $\F_5$, $\F_{11}$ are called {\it main classes}.

In this paper we consider 3-dimensional almost contact B-metric manifolds $(M,\varphi,\xi ,\eta ,g)$ belonging to the class $\F_4$, which is determined by (see \cite{GaMGri})
\begin{align}\label{2.2}
\begin{array}{ll}
\F_4 : F(X,Y,Z)=-\frac{\theta (\xi )}{2}\{g(\varphi X,\varphi Y)\eta (Z)+g(\varphi X,\varphi Z)\eta (Y)\}. 
\end{array}
\end{align}
Taking into account \eqref{2.1} and \eqref{2.2} we have 
\begin{align}\label{2.3}
  \nabla _X\xi= \frac{\theta (\xi )}{2}\varphi X .
\end{align}
The equality \eqref{2.3} shows that the class $\F_4$ is similar to the class of $\alpha $-Sasakian manifolds in the theory of almost contact metric manifolds.\\
Let ${\widetilde \nabla }$ be the Levi-Civita connection of ${\widetilde g}$. We consider the symmetric tensor field $\Phi $ of type $(1,2)$ defined by
$\Phi (X,Y)={\widetilde \nabla }_XY-\nabla _XY$. For a 3-dimensional $\F_4$-manifold the following equality holds \cite{MM}: 
\begin{align}\label{2.4}
{\widetilde \nabla }_XY-\nabla _XY=\frac{\theta (\xi )}{2}\{g(X,\varphi Y)-g(\varphi X,\varphi Y)\}\xi .
\end{align}
Let us remark that on a 3-dimensional almost contact B-metric manifold $(M,\varphi,\xi,\eta,g)$ the metric $g$ has signature $(2,1)$, \ie $(M,g)$ is a 3-dimensional Lorentzian manifold. In the remainder of this section we briefly recall the main notions about null curves in a 3-dimensional Lorentzian manifold $M$ for which we refer to \cite{D-B, D-J}.

Let 
$C: I\longrightarrow M$ be a smooth curve in $M$ given locally by
\[
x_i=x_i(t), \quad t\in I\subseteq {\R}, \quad i\in \{1,2,3\}
\]
for a coordinate neighborhood $U$ of $C$. The tangent vector field is given by
\[
\frac{{\rm d}}{{\rm d}t}=(\dot {x}_1, \dot {x}_2, \dot {x}_3)=\dot {C},
\]
where we denote $\frac{{\rm d}x_i}{{\rm d}t}$ by $\dot {x}_i$ for $i\in \{1,2,3\}$. The curve $C$ is called a {\it regular curve} if $\dot {C}\neq 0$ holds everywhere.

Let a regular curve $C$ be {\it a null (lightlike) curve} in $(M, g)$, \ie at each point $x$ of $C$ we have
\begin{align}\label{2'}
g(\dot {C},\dot {C})=0,\qquad \dot {C}\neq 0.
\end{align}
%
A general Frenet frame on $M$ along $C$ is denoted by ${\bf F}=\{\dot {C}, N, W\}$ and the vector fields in ${\bf F}$ are determined by
\begin{align}\label{3'}
g(\dot {C},N)=g(W,W)=1, \quad g(N,N)=g(N,W)=g(\dot {C},W)=0.
\end{align}
In \cite[Theorem 1.1, p. 53]{D-B} it was proved that for a given $W$ there exists a unique $N$ satisfying
the corresponding equalities in \eqref{3'}.
The following general Frenet equations with respect to ${\bf F}$ and $\nabla $ of $(M, g)$ are known from \cite{D-B, D-J} %
\begin{align}\label{general Frenet eq}
\begin{array}{lll}
\nabla _{\dot {C}}\dot {C}=h\dot {C}+k_1W \\
\nabla _{\dot {C}}N=-hN+k_2W \\
\nabla _{\dot {C}}W=-k_2\dot {C}-k_1N,
\end{array}
\end{align}
where 
$h$, $k_1$ and  $k_2$
are smooth functions on $U$. The functions $k_1$ and  $k_2$ are called {\it curvature functions} of $C$.

The general Frenet frame ${\bf F}$ and its general Frenet equations \eqref{general Frenet eq} are not unique as they depend on the parameter and the choice of the
screen vector bundle $S(TC^\bot )={\rm span}W$ of $C$ (for details see \cite[pp. 56-58]{D-B}, \cite[pp. 25-29]{D-J}). It is known \cite[p. 58]{D-B} that there exists a
parameter $p$ called a {\it distinguished parameter}, for which the function $h$ vanishes in \eqref{general Frenet eq}. The pair $(C(p), {\bf F})$, where ${\bf F}$ is a Frenet frame along $C$ with respect to a distinguished parameter $p$, is called a {\it framed null curve} (see \cite{D-J}). In general, $(C(p), {\bf F})$ is not unique since it depends on both $p$ and the screen distribution. A Frenet frame with the minimum number of curvature functions is called {\it Cartan Frenet frame} of a null curve $C$. In \cite{D-J} it is proved that if the null curve $C(p)$ is non-geodesic
such that  for $\ddot{C}=\frac{{\rm d}}{{\rm d}p}\dot{C}$ the condition $g(\ddot {C},\ddot {C})=k_1=1$ holds, then there exists only one Cartan Frenet frame
${\bf F}$ 
with the following Frenet equations
\begin{align}\label{Cartan Frenet eq}
\begin{array}{lll}
\nabla _{\dot {C}}\dot {C}=W, \quad
\nabla _{\dot {C}}N=\tau W, \quad
\nabla _{\dot {C}}W=-\tau \dot {C}-N.
\end{array}
\end{align}
The latter equations are called the {\it Cartan Frenet equations} of $C(p)$ whereas $\tau $ is called a \emph{torsion function} and it is invariant upto a sign under Lorentzian transformations. A null curve together with its Cartan Frenet frame is called a {\it Cartan framed null curve}. Note that some authors \cite{H-I} term a framed null curve  $(C(p), {\bf F})$ Cartan framed null curve and a Frenet frame 
${\bf F}$ along $C$ with respect to a distinguished parameter $p$ -- Cartan Frenet frame.  

\section{Framed $\varphi $-slant null curves with respect to the original parameter in 3-dimensional
${\F}_4$-manifolds}\label{sec-3}
Let us consider a smooth curve $C$ 
in an almost contact B-metric manifold $(M,\varphi,\xi,\eta,g)$. We say
that $C$ is a {\it slant curve} on $M$ if $g(\dot {C},\xi )=\eta (\dot {C})=a$ and $a$ is a real constant. The curve $C$ is called a {\it Legendre curve} if $a=0$.\\
A distinguishing feature of the almost contact B-metric manifolds from the almost contact metric and almost paracontact metric manifolds is that $g(X,\varphi X)$ is not zero in general. Motivated by this fact we define a new type slant curves.
\begin{defn}\label{Definition 3.1}
We say that a smooth curve $C(t)$ in an almost contact B-metric manifold $(M,\varphi,\xi,\eta,g)$ is {\it 
$\varphi $-slant} if 
\begin{align}\label{3.1}
g(\dot {C}(t),\xi )=\eta (\dot {C}(t))=a \quad \text{and}\quad g(\dot {C}(t),\varphi \dot {C}(t))=b,
\end{align}
 where $a$ and $b$ are real constants.
\end{defn}
\begin{rem}\label{Remark 3.1}
Let $C(t)$ be a slant or a $\varphi $-slant curve. 
If we change the parameter $t$ of $C(t)$ with another parameter $p$, then we have $\dot C(p)=\dot C(t)\frac{{\rm d}t}{{\rm d}p}$. Hence \eqref{3.1} becomes
\begin{align*}
\begin{array}{ll}
g(\dot {C}(p),\xi )=\eta (\dot {C}(p))=\frac{{\rm d}t}{{\rm d}p}\eta (\dot {C}(t))=\frac{{\rm d}t}{{\rm d}p}a \quad \text{and} \\ \\
g(\dot {C}(p),\varphi \dot {C}(p))={\left(\frac{{\rm d}t}{{\rm d}p}\right)}^2g(\dot {C}(t),\varphi \dot {C}(t))={\left(\frac{{\rm d}t}{{\rm d}p}\right)}^2b .
\end{array}
\end{align*}
Therefore $g(\dot C(p),\xi )$ and $g(\dot {C}(p),\varphi \dot {C}(p))$ are constant
if and only if $t=\alpha p+\beta $, where $\alpha, \, \beta $ are constant, i.e in general slant and $\varphi $-slant curves are not invariant under a reparameterization.
Our aim in the present paper is to study $\varphi $-slant null curves with respect to its original parameter.
\end{rem}

\begin{prop}\label{Proposition 3.2}
In a 3-dimensional almost contact B-metric manifold \\ $(M,\varphi,\xi,\eta,g)$ there exist no $\varphi $-slant null curves such that $a=b=0$.
\end{prop}
\begin{proof} Let us assume that there exists a $\varphi $-slant null curve $C$ in $M$ such that $a=b=0$.
Then we have $g(\varphi \dot {C},\varphi \dot {C})=-g(\dot {C}, \dot {C})+
\eta (\dot {C})\eta (\dot {C})=0$. From $\eta (\dot {C})=\eta (\varphi \dot {C})=0$ it follows that $\dot {C}$ and $\varphi \dot {C}$ belong to the contact distribution $\mathbb {D}$ of $M$ along $C$. Since 
$\dot {C}$ and $\varphi \dot {C}$ are linearly independent, they form a basis of $\mathbb {D}$ at each point $x$ of $C$. Hence, for an arbitrary vector field $X \in \mathbb {D}_{\vert C}$ we have $X=u\dot {C}+v\varphi \dot {C}$ for some functions $u$ and $v$. By using $g(\dot {C},\dot {C})=g(\varphi \dot {C},\varphi \dot {C})=0$ and the second equality in \eqref{3.1} we obtain $g(X,X)=0$. The last implies a contradiction since $g_{\vert \mathbb D}$ is of signature $(1,1)$, which confirms our assertion.
\end{proof}
Now, taking into account Proposition \ref{Proposition 3.2}, it is easy to see that the triad of vector fields
$\{\dot {C}, \xi, \varphi \dot {C}\}$ is a basis of $T_xM$ at each point $x$ of a $\varphi $-slant null curve in a 3-dimensional almost contact B-metric manifold $(M,\varphi,\xi,\eta,g)$. By using this basis, 
in \cite{M-N} H. Manev and the author of this paper obtained the following result for a slant null curve $C$ in $(M,\varphi,\xi,\eta,g)$ satisfying the conditions \eqref{3.1}, where $(a,b)\neq (0,0)$ and b is a function:\\
If ${\bf F}=\{\dot {C}, N, W\}$ is a general Frenet frame on $M$ along $C$ which has the same positive orientation as a basis $\{\dot {C}, \xi, \varphi \dot {C}\}$ at
each $x\in C$, then
\begin{align}\label{3.21}
W=\alpha \xi +\beta \dot {C}+\gamma \varphi \dot {C} , 
\end{align}
\begin{align}\label{3.22}
N=\lambda \xi+\mu \dot {C}+\nu \varphi \dot {C},
\end{align}
where $\beta $ is an arbitrary function and $\alpha, \gamma, \lambda, \mu, \nu $ are the following functions
\begin{align}\label{3.3}
\begin{split}
\alpha ={-\frac{b}{\sqrt{a^4+b^2}}},\qquad
\gamma ={\frac{a}{\sqrt{a^4+b^2}}},
\end{split}
\end{align}
\begin{align}\label{3.4}
\begin{split}
\lambda &=\frac{a^3+\beta b\sqrt{a^4+b^2}}
{a^4+b^2}, \qquad
\mu =-\frac{a^2+\beta ^2\left(a^4+b^2\right)}{2\left(a^4+b^2\right)}, 
\\[4pt]
\nu &=\frac{b-\beta a
\sqrt{a^4+b^2}}{a^4+b^2}.
\end{split}
\end{align}
Moreover, the functions $h$ and $k_1$ in \eqref{general Frenet eq}
with respect to ${\bf F}$ are given by
\begin{align}\label{3.5}
\begin{array}{ll}
h=-\lambda g(\dot {C},\nabla _{\dot {C}}\xi )+\frac{\nu }{2}\left[\dot {C}\left(b\right)-F(\dot {C},\dot {C},\dot {C})\right],\\ \\
k_1=-\alpha g(\dot {C},\nabla _{\dot {C}}\xi )+\frac{\gamma }{2}\left[\dot {C}\left(b\right)-F(\dot {C},\dot {C},\dot {C})\right].
\end{array}
\end{align}
\begin{rem}\label{Remark 3.2}
The equalities \eqref{3.21}, \eqref{3.22}, \eqref{3.3}, \eqref{3.4} and \eqref{3.5} hold also in case of a 
$\varphi $-slant null curve in a 3-dimensional $(M,\varphi,\xi,\eta,g)$, i.e. when b is a constant.
\end{rem}
\begin{prop}\label{Proposition 3.3}
Let $C$ be a $\varphi $-slant null curve in a 3-dimensional $\F_4$-manifold $(M,\varphi,\xi,\eta,g)$. If 
${\bf F}=\{\dot {C}, N, W\}$ is a general Frenet frame along $C$, then for the functions $k_1$ and $h$ in \eqref{general Frenet eq} we have
\begin{align}\label{3.61}
h=-\beta \frac{\theta (\xi )\sqrt{a^4+b^2}}{2} ,
\end{align}
\begin{align}\label{3.62}
k_1=\frac{\theta (\xi )\sqrt{a^4+b^2}}{2}.
\end{align}
\end{prop}
\begin{proof}
By using \eqref{2.2} and \eqref{2.3} we find
\begin{align}\label{3.66}
F(\dot {C},\dot {C},\dot {C})=-\theta (\xi )a^3\quad \text{and} \quad 
g(\dot {C},\nabla _{\dot {C}}\xi )=\frac{\theta (\xi )b}{2} .
\end{align}
Substituting the above equalities and $\dot {C}\left(b\right)=0$ in \eqref{3.5}, we get 
\begin{align}\label{3.7}
h=\frac{\theta (\xi )}{2}(-\lambda b+\nu a^3), \qquad k_1=\frac{\theta (\xi )}{2}(-\alpha b+a^3).
\end{align}
By virtue of \eqref{3.3} and \eqref{3.4} we obtain
\[
-\lambda b+\nu a^3=-\beta \sqrt{a^4+b^2}, \qquad -\alpha b+a^3=\sqrt{a^4+b^2}.
\]
The latter equalities and \eqref{3.7} imply \eqref{3.61} and \eqref{3.62}.
\end{proof}
\begin{cor}\label{Corollary 3.4}
A $\varphi $-slant null curve $C$ in a 3-dimensional $\F_4$-manifold $M$ is geodesic if and only if $M$ is an $\F_0$-manifold.
\end{cor}
\begin{proof}
It is known \cite{D-B} that a null curve is geodesic if and only if $k_1$ vanishes. From \eqref{3.62} and $a^4+b^2\neq 0$ it follows that $k_1=0$ if and only if $\theta (\xi )=0$ that is 
$M\in \F_0$.
\end{proof}
\begin{thm}\label{Theorem 3.5}
Let $C(t)$ be a non-geodesic $\varphi $-slant null curve in a 3-dimensional ${\F}_4$-manifold $(M,\varphi,\xi,\eta,g)$. Then there exists a unique Frenet frame ${\bf F_1}=\{\dot {C}, N_1, W_1\}$ for which the original parameter $t$ of $C(t)$ is distinguished and the vector fields $W_1$, $N_1$ are given by
\begin{align}\label{3.8}
\begin{array}{lll}
\displaystyle W_1=\alpha \xi +\gamma \varphi \dot {C}=-\frac{b}{\sqrt{a^4+b^2}}\xi +\frac{a}{\sqrt{a^4+b^2}}\varphi \dot {C} , \\ \\
\displaystyle N_1=\lambda _1\xi +\mu _1\dot {C}+\nu _1\varphi \dot {C}\\
\quad \, \, \, =\displaystyle \frac{a^3}{a^4+b^2}\xi -\frac{a^2}{2(a^4+b^2)}\dot {C}+\frac{b}{a^4+b^2}\varphi \dot {C}.
\end{array}
\end{align}
The function $k_2$ with respect to ${\bf F_1}$ is given by
\begin{align}\label{3.9}
k_2=\frac{a^2\theta (\xi )}{4\sqrt{a^4+b^2}} .
\end{align}
\end{thm}
\begin{proof}
By using \eqref{3.61} and \eqref{3.62} we have $h=-\beta k_1$. Then, having in mind \eqref{3.21}, the first equality in \eqref{general Frenet eq} becomes
\begin{align}\label{3.99}
\nabla _{\dot C}\dot C=h\dot C+k_1W=-\beta k_1\dot C+k_1(\alpha \xi +\beta \dot {C}+\gamma  \varphi \dot {C})=k_1(\alpha \xi +\gamma  \varphi \dot {C}).
\end{align}
The vector field $W_1=\alpha \xi +\gamma  \varphi \dot {C}=\displaystyle-\frac{b}{\sqrt{a^4+b^2}}\xi +\frac{a}{\sqrt{a^4+b^2}}\varphi \dot {C}$ is obtained from \eqref{3.21} for $\beta =0$ and therefore $g(W_1,W_1)=1$. Replacing  $\beta $ with $0$  in \eqref{3.4} we get 
\begin{align}\label{3.11}
\lambda _1=\frac{a^3}{a^4+b^2}, \quad \mu _1=-\frac{a^2}{2(a^4+b^2)}, \quad \nu _1=\frac{b}{a^4+b^2}.
\end{align}
Hence, the unique vector field $N_1$ corresponding to $W_1$ is given by $N_1=\lambda _1\xi +\mu _1\dot {C}+\nu _1\varphi \dot {C}$. Comparing \eqref{3.99} with the first equality in \eqref{general Frenet eq} we infer that $h=0$ with respect to the Frenet frame ${\bf F_1}=\{\dot {C}, N_1, W_1\}$, where $W_1$ and $N_1$ are determined by \eqref{3.8}. Thus, the original parameter $t$ of $C(t)$ is distinguished with respect to ${\bf F_1}$. Now, let we take another Frenet frame ${\bf F^*}=\{\dot {C}, N^*, W^*\}$ along $C$ with respect to $t$ and $W^*$. Since for a given $C$ the vector field $W$ depends only on $\beta $, we have $W^*=\alpha \xi +\beta ^*\dot {C}+\gamma  \varphi \dot {C}=
W_1+\beta ^*\dot {C}$. The unique vector field $N^*$ corresponding to $W^*$ is given by
$N^*=\lambda ^*\xi +\mu ^*\dot {C}+\nu ^*\varphi \dot {C}$, where $\lambda ^*$, $\mu ^*$ and
$\nu ^*$ are obtained replacing $\beta $ with $\beta ^*$ in \eqref{3.4}. For the first equality in \eqref{general Frenet eq} with respect to ${\bf F^*}$ we have
\begin{align*}
\nabla _{\dot C}\dot C=h^*\dot C+k_1^*W^* ,
\end{align*}
where $h^*=-\beta ^*k_1^*$. From the above equality we find $k_1^*=g(\nabla _{\dot C}\dot C,
W^*)=g(\nabla _{\dot C}\dot C,W_1)=k_1$. Hence $h^*=-\beta ^*k_1$. The parameter $t$ is  distinguished with respect to ${\bf F^*}$ if and only if $h^*$ vanishes.
Since $C(t)$ is  non-geodesic, it follows that $h^*=0$ if and only if $\beta ^*=0$. Thus, ${\bf F^*}={\bf F_1}$.\\
From the second equality in \eqref{general Frenet eq} with respect to ${\bf F_1}$  we derive 
\[
k_2=g(\nabla _{\dot {C}}N_1,W_1).
\]
Taking into account \eqref{3.8}, the latter equality becomes
\begin{align}\label{3.12}
\begin{array}{ll}
k_2=\alpha \left(\lambda _1g(\nabla _{\dot {C}}\xi ,\xi )+\mu _1g(\nabla _{\dot {C}}\dot {C},\xi )+
\nu _1g(\nabla _{\dot {C}}\varphi \dot {C},\xi )\right)\\
\qquad +\gamma \left(\lambda _1g(\nabla _{\dot {C}}\xi ,\varphi \dot {C})+\mu _1g(\nabla _{\dot {C}}\dot {C},\varphi \dot {C})+
\nu _1g(\nabla _{\dot {C}}\varphi \dot {C},\varphi \dot {C})\right).
\end{array}
\end{align}
The equalities $g(\xi ,\xi )=1$ and $g(\varphi \dot {C},\varphi \dot {C})=a^2$ imply
\begin{align}\label{3.13}
g(\nabla _{\dot {C}}\xi ,\xi )=g(\nabla _{\dot {C}}\varphi \dot {C},\varphi \dot {C})=0 .
\end{align}
By using $g(\dot {C} ,\xi )=a$, $g(\varphi \dot {C} ,\xi )=0$ and \eqref{2.3} we receive
\begin{align}\label{3.14}
\begin{array}{ll}
g(\nabla _{\dot {C}}\dot {C},\xi )=-g(\dot {C},\nabla _{\dot {C}}\xi )=\displaystyle -\frac{\theta (\xi )b}{2} ,\\
g(\nabla _{\dot {C}}\varphi \dot {C},\xi )=-g(\nabla _{\dot {C}}\xi ,\varphi \dot {C})=\displaystyle -\frac{\theta (\xi )a^2}{2} .
\end{array}
\end{align}
With the help of the following expressions 
\begin{align*}
\begin{array}{ll}
F(\dot {C},\dot {C},\dot {C})=g(\nabla _{\dot {C}}\dot {C},\varphi \dot {C})-g(\varphi (\nabla _{\dot {C}}\dot {C}),\dot {C}) ,\\ \\
0=\dot {C}(b)=g(\nabla _{\dot {C}}\dot {C},\varphi \dot {C})+g(\dot {C},\nabla _{\dot {C}}\varphi \dot {C})
\end{array}
\end{align*}
and the first equality in \eqref{3.66} we find
\begin{align}\label{3.15}
g(\nabla _{\dot {C}}\dot {C},\varphi \dot {C})=-\frac{1}{2}F(\dot {C},\dot {C},\dot {C})=
\frac{1}{2}\theta (\xi )a^3 .
\end{align}
Substituting \eqref{3.13}, \eqref{3.14} and \eqref{3.15} in \eqref{3.12} we obtain
\begin{align}\label{3.16}
k_2=-\frac{\alpha \theta (\xi )}{2}(\mu _1b+\nu _1a^2)+\frac{\gamma \theta (\xi )a^2}{2}(\lambda _1+\mu _1a) .
\end{align}
Finally, substituting \eqref{3.3} and \eqref{3.11} in \eqref{3.16} we get \eqref{3.9}.
\end{proof}
From now on in this paper, we deal with the pair $(C(t), {\bf F_1})$, where $C(t)$ is a $\varphi $-slant null curve in a 3-dimensional $\F_4$-manifold for which the original parameter is distinguished and ${\bf F_1}$ is the unique Frenet frame of $C(t)$ from \thmref{Theorem 3.5}. The Frenet equations of $(C(t), {\bf F_1})$ are
\begin{align}\label{3.17}
\begin{array}{lll}
\nabla _{\dot {C}}\dot {C}=k_1W_1 \\
\nabla _{\dot {C}}N_1=k_2W_1 \\
\nabla _{\dot {C}}W_1=-k_2\dot {C}-k_1N_1,
\end{array}
\end{align}
where $k_1$ and  $k_2$ are given by \eqref{3.62} and \eqref{3.9}, respectively.
\begin{defn}
A framed null curve with $k_2=0$ is called {\it a generalized null cubic}.
\end{defn}
Substituting $a=0$ in \eqref{3.62},  \eqref{3.8} and \eqref{3.9}, we state
\begin{prop}\label{Proposition 3.6}
Let $(C(t), {\bf F_1})$ be a Legendre $\varphi $-slant null curve  in a 3-dimensional $\F_4$-manifold $(M,\varphi,\xi,\eta,g)$. Then we have
\par
\item (i) $k_1=\displaystyle\frac{\vert b\vert\theta (\xi )}{2}$.
\par
\item (ii) The vector fields $N_1$, $W_1$ from ${\bf F_1}=\{\dot {C}, N_1, W_1\}$ are given by
 $N_1=\frac{1}{b}\varphi \dot {C}$, $W_1=-\epsilon \xi $, where  
$\epsilon =\{{\rm sign} \, b\}=\{\pm 1\}$.
\par
\item (iii) $(C(t), {\bf F_1})$  is a generalized null cubic.
\end{prop}
As a generalization of the magnetic curves in \cite{Bejan} was introduced the notion of $F$-geodesics in a manifold $M$ endowed with a (1,1)-tensor field $F$ and with a linear connection $\nabla $.
\begin{defn}\label{Definition 3.2}\cite{Bejan}
A smooth curve $\gamma : I\longrightarrow M$ in a manifold $(M,F,\nabla )$ is an $F$-geodesic if 
$\gamma (t)$ satisfies $\nabla _{\dot {\gamma }(t)}\dot {\gamma }(t)=F\dot {\gamma }(t)$.
\end{defn}
Note that an $F$-geodesic is not invariant under a reparameterization.\\
Using \eqref{3.62} and \eqref{3.8} the first equality in \eqref{3.17} becomes
\begin{align}\label{3.18}
\nabla _{\dot {C}}\dot {C}=\displaystyle-\frac{b\theta (\xi )}{2}\xi +
\displaystyle\frac{a\theta (\xi )}{2}\varphi \dot {C} .
\end{align}
By virtue of \eqref{3.18} we establish the truth of the following
\begin{prop}\label{Proposition 3.7}
A  $\varphi $-slant null curve $(C(t), {\bf F_1})$ in a 3-dimensional $\F_4$-manifold $(M,\varphi,\xi,\eta,g)$ is a $\varphi $-geodesic if and only if $b=0$ and $\theta (\xi )=\frac{a}{2}$.
\end{prop}

\section{Some non-null $\varphi $-slant curves in a 3-dimensional $\F_4$-manifold induced from $\varphi $-slant null curves}\label{sec-4}
A curve $\gamma : I\longrightarrow M$ in a 3-dimensional Lorentzian manifold $(M,g)$ is said to be {\it a unit speed curve} (or $\gamma $ is parameterized by arc length $s$) if $g(\gamma ^\prime ,\gamma ^\prime)=\epsilon _1=\pm 1$, where $\gamma ^\prime =\frac{{\rm d}\gamma }{{\rm d}s}$ is the velocity vector field. A unit speed curve 
$\gamma $ is said to be {\it spacelike} or {\it timelike} if $\epsilon _1=1$ or $\epsilon _1=-1$, respectively. A unit speed curve $\gamma $ is said to be {\it a Frenet curve} if 
one of the following three cases holds \cite{W}:
\begin{itemize}
\item $\gamma $ is of osculating order 1 that is $\nabla _{\gamma ^\prime }\gamma ^\prime =0$, i.e. 
$\gamma $ is a geodesic;
\item $\gamma $ is of osculating order 2, i.e. there exist two orthonormal vector fields $E_1$, $E_2$ and a positive function $k$ (the curvature) along $\gamma $ such that $E_1=\gamma ^\prime $, $g(E_2,E_2)=\epsilon _2=\pm 1$ and
\[
\nabla _{\gamma ^\prime }E_1=\epsilon _2kE_2 , \quad \nabla _{\gamma ^\prime }E_2=-\epsilon _1kE_1; 
\]
\item $\gamma $ is of osculating order 3, i.e. there exist three orthonormal vector fields $E_1$, $E_2$, 
$E_3$ and two positive functions $k$ (the curvature) and $\tau $ (the torsion) along $\gamma $ such that 
$E_1=\gamma ^\prime $, $g(E_2,E_2)=\epsilon _2=\pm 1$, $g(E_3,E_3)=\epsilon _3=\pm 1$, 
$\epsilon _3=-\epsilon _1\epsilon _2$ and 
\[
\nabla _{\gamma ^\prime }E_1=\epsilon _2kE_2 , \quad \nabla _{\gamma ^\prime }E_2=-\epsilon _1kE_1+\epsilon _3\tau E_3 , \quad \nabla _{\gamma ^\prime }E_3=-\epsilon _2\tau E_2. 
\]
\end{itemize} 
As in the case of Riemannian geometry, a Frenet curve  in a 3-dimensional Lorentzian manifold is a geodesic if and only if its curvature $k$ vanishes. Also a curve $\gamma $ with a curvature $k$ and a torsion $\tau $ is called \cite{I}:
\begin{itemize}
\item {\it a pseudo-circle} if $k=const$ and $\tau =0$;
\item {\it a helix} if $k=const$ and $\tau =const$;
\item {\it a proper helix} if $\gamma $ is a helix which is not a circle;
\item {\it a  generalized helix} if $\displaystyle\frac{k}{\tau}=const$ but $k$ and $\tau $ are not constant.
\end{itemize}
Taking into account  Remark \ref{Remark 3.1} we state:\\
A Frenet curve $\gamma (s)$ in an almost contact B-metric manifold $(M,\varphi,\xi,\eta,g)$
is said to be slant if $\eta (\gamma ^\prime (s))=a=const$. (see \cite{W})\\
We say that a Frenet curve $\gamma (s)$ in an almost contact B-metric manifold $(M,\varphi,\xi,\eta,g)$ is  
$\varphi $-slant if 
\begin{align*}
\eta (\gamma ^\prime (s))=a=const \quad \text{and}\quad g(\gamma ^\prime (s),\varphi \gamma ^\prime (s))=b=const.
\end{align*}
Since there exist two B-metrics $g$ and $\widetilde g$ on an almost contact B-metric manifold $M$, we can consider a curve $\gamma $ in $M$ with respect to both $g$ and $\widetilde g$. In this section we investigate non-null curves with respect to  $\widetilde g$ induced from two types of $\varphi $-slant null curves with respect to $g$ in a 3-dimensional $\F_4$-manifold.
\begin{thm}\label{Theorem 4.1}
Let $(C(t), {\bf F_1})$ be a Legendre $\varphi $-slant null curve with respect to $g$ in a 3-dimensional 
$\F_4$-manifold. The curve $C$ with respect to $\widetilde g$ is
\par
\item (i) spacelike if $b>0$ or timelike if $b<0$;
\par
(ii) a Legendre $\varphi $-slant curve;
\par
(iii) a geodesic.
\end{thm}
\begin{proof}
(i) Since $(C(t), {\bf F_1})$ is a Legendre $\varphi $-slant null curve, from Proposition \ref{Proposition 3.2}
it follows that $b\neq 0$. Thus $\widetilde g(\dot C,\dot C)=b\neq 0$. Now, we parameterize $C(t)$ by its 
arc length parameter $\widetilde s$ with respect to $\widetilde g$ given by
\[
\widetilde s=\int^{t}_{0}\sqrt{\vert \widetilde g(\dot C,\dot C)\vert }{\rm d}u=\int^{t}_{0}\sqrt{\vert b\vert }{\rm d}u =\sqrt{\vert b\vert }t .
\]
Then for the tangent vector $C^\prime (\widetilde s)=\displaystyle\dot C(t)\frac{{\rm d}t}{{\rm d}\widetilde s}=\frac{\dot C(t)}{\sqrt{\vert b\vert }}$ of the curve $C(\widetilde s)$ we have 
$\widetilde g(C^\prime ,C^\prime )=\displaystyle\frac{b}{\vert b\vert }=\pm 1$ which confirms the assertion (i).
\par
(ii) By direct calculations we find
\begin{align}\label{4.1}
\begin{array}{ll}
\widetilde \eta (C^\prime )=\widetilde g(C^\prime ,\xi )=\displaystyle\eta (C^\prime )=\frac{1}{\sqrt{\vert b\vert }}\eta (\dot C)=0 , \\ \\
\displaystyle\widetilde g(C^\prime ,\varphi  C^\prime )=\frac{1}{\vert b\vert }\widetilde g(\dot C,\varphi \dot C)=0 .
\end{array}
\end{align}
The equalities \eqref{4.1} show that $C(\widetilde s)$ is a Legendre $\varphi $-slant curve.\\
(iii) By virtue of \eqref{2.4} we get
\begin{align}\label{4.2}
\widetilde \nabla _{C^\prime }C^\prime =\frac{1}{\vert b\vert }\widetilde \nabla _{\dot C}\dot C=
\frac{1}{\vert b\vert }\left(\nabla _{\dot C}\dot C+\frac{b\theta (\xi )}{2}\xi \right).
\end{align}
Now, we find $\nabla _{\dot C}\dot C$ with the help of Proposition \ref{Proposition 3.6}
\begin{align*}
\nabla _{\dot C}\dot C=k_1W_1=\frac{\vert b\vert \theta (\xi )}{2}(-\epsilon \xi )=-\frac{b\theta (\xi )}{2}\xi .
\end{align*}
The latter equality and \eqref{4.2} imply $\widetilde \nabla _{C^\prime }C^\prime =0$, i.e. $C(\widetilde s)$ is a geodesic.
\end{proof}
\begin{thm}\label{Theorem 4.2}
Let $(C(t), {\bf F_1})$ be a $\varphi $-slant null curve with respect to $g$ in a 3-dimensional 
$\F_4$-manifold $M$ and $b=0$. The curve $C$ with respect to $\widetilde g$ is a $\varphi $-slant spacelike curve in $M$ such that:
\par
\item (i) If $(C(t), {\bf F_1})$ is non-geodesic, then $C$ is a Frenet curve of osculating order 3. The orthonormal vector fields $E_1$, $E_2$, $E_3$ with respect to $\widetilde g$, the curvature $\widetilde k$ and the torsion $\widetilde \tau $ along $C$ are given as follows:
\begin{align}\label{4.3}
E_1(\widetilde s)=C^\prime (\widetilde s)=\frac{\dot C}{\vert a \vert}, \quad \widetilde g(E_1,E_1)=\epsilon _1=1 ,
\end{align}
where $\widetilde s$ is the arc length parameter of $C(t)$ with respect to $\widetilde g$;
\begin{align}\label{4.4}
E_2(\widetilde s)=\widetilde \epsilon \left(\frac{1}{a} \varphi \dot C-\xi \right)=\widetilde \epsilon \left(-\frac{1}{2a}\dot C-aN_1+W_1\right), 
\, \widetilde g(E_2,E_2)=\epsilon _2=1,
\end{align}
where $\widetilde \epsilon =\{{\rm sign} \, k_1(t)\}=\{{\rm sign}\, \theta (\xi )(t)\}=\{\pm 1\}$, 
$\epsilon =\{{\rm sign} \, a\}=\{\pm 1\}$ and $ k_1(t)$ is the curvature of $(C(t), {\bf F_1})$;
\begin{align}\label{4.5}
E_3(\widetilde s)=\frac{1}{\vert a\vert}(\varphi \dot C-\dot C)=\epsilon \left(-\frac{1}{a}\dot C+W_1\right), \quad \widetilde g(E_3,E_3)=\epsilon _3=-1;
\end{align}
\begin{align}\label{4.6}
\widetilde k(\widetilde s)=\frac{\vert k_1(t)\vert }{a^2}=\frac{\vert \theta (\xi )(t)\vert }{2}, \quad \widetilde \tau (\widetilde s)=k(\widetilde s).
\end{align}
(ii) If $(C(t), {\bf F_1})$ is a geodesic, then $C(\widetilde s)$ is also a geodesic.
\end{thm}
\begin{proof}
Since $b=0$ for $(C(t), {\bf F_1})$, from Proposition \ref{Proposition 3.2} it follows that $a\neq 0$.
First, for further use we compute:
\begin{align}\label{4.7}
\begin{array}{lll}
\widetilde g(\dot C,\dot C)=g(\dot C,\varphi \dot C)+(\eta (\dot C))^2=a^2 ,\, \, \widetilde g(\dot C,\varphi \dot C)=g(\varphi \dot C,\varphi \dot C)=a^2,\\
\widetilde g(\varphi \dot C,\varphi \dot C)=-\widetilde g(\dot C,\dot C)+(\eta (\dot C))^2=0. 
\end{array}
\end{align}
The curvature $k_1(t)$ and the vector fields $W_1$, $N_1$ from the frame ${\bf F_1}$ along $C(t)$ we obtain by substituting $b=0$ in \eqref{3.62} and \eqref{3.8}. Thus we have 
\begin{align}\label{4.8}
k_1(t)=\frac{a^2\theta (\xi )(t)}{2} ,
\end{align}
\begin{align}\label{4.9}
W_1=\frac{1}{a}\varphi \dot C ,
\end{align}
\begin{align}\label{4.10}
N_1=\frac{1}{a}\xi -\frac{1}{2a^2}\dot C .
\end{align}
Since $\widetilde g(\dot C,\dot C)=a^2\neq 0$, analogously as in the proof of Theorem \ref{Theorem 4.1} we parameterize $C(t)$ with respect to its arc length parameter $\widetilde s=\vert a\vert t$. Then it is easy to see that for the vector field $E_1(\widetilde s)=C^\prime (\widetilde s)$ the equality $\widetilde g(E_1,E_1)=1$ holds. Hence $C(\widetilde s)$ is a spacelike curve with respect to $\widetilde g$. By straightforward calculations we obtain
\begin{align*}
\begin{array}{ll}
\widetilde \eta (C^\prime )=\widetilde g(C^\prime ,\xi )=\displaystyle\eta (C^\prime )=\frac{1}{\vert a\vert }\eta (\dot C)=\frac{a}{\vert a\vert }=\epsilon , \\ \\
\displaystyle\widetilde g(C^\prime ,\varphi  C^\prime )=\frac{1}{a^2}\widetilde g(\dot C,\varphi \dot C)=\frac{1}{a^2}a^2=1 .
\end{array}
\end{align*}
From the above equalities it is clear that the spacelike curve $C(\widetilde s)$ is a $\varphi $-slant (non-Legendre) curve in $M$.\\
(i) By virtue of \eqref{2.4} we find
\begin{align}\label{4.11}
\widetilde \nabla _{C^\prime }C^\prime =\frac{1}{a^2}\widetilde \nabla _{\dot C}\dot C=
\frac{1}{a^2}\left(\nabla _{\dot C}\dot C-\frac{a^2\theta (\xi )}{2}\xi \right).
\end{align}
From the first equality in \eqref{3.17} and \eqref{4.9} we get
\begin{align*}
\nabla _{\dot C}\dot C=k_1W_1=k_1\frac{1}{a}\varphi \dot C .
\end{align*}
We substitute the latter equality in \eqref{4.11}. Then having in mind \eqref{4.8} we obtain
\begin{align*}
\widetilde \nabla _{C^\prime }C^\prime =\frac{k_1}{a^2}\left(\frac{1}{a}\varphi \dot C-\xi \right) .
\end{align*}
We rewrite the above equality in the following equivalent form
\begin{align}\label{4.12}
\widetilde \nabla _{C^\prime }C^\prime =\frac{\widetilde \epsilon k_1}{a^2}\widetilde \epsilon \left(\frac{1}{a}\varphi \dot C-\xi \right)=\frac{\widetilde \epsilon k_1}{a^2}E_2(\widetilde s),
\end{align}
where we put $E_2(\widetilde s)=\widetilde \epsilon \left(\frac{1}{a}\varphi \dot C-\xi \right)$ and 
$\widetilde \epsilon =\{{\rm sign} \, k_1(t)\}=\{{\rm sign}\, \theta (\xi )(t)\}=\{\pm 1\}$. By direct calculations we check that $\widetilde g(E_2,E_2)=1$ and $\widetilde g(E_1,E_2)=0$. From \eqref{4.10} 
we derive $\xi =aN_1+\frac{1}{2a}\dot C$ and hence $E_2=\widetilde \epsilon \left(-\frac{1}{2a}\dot C-aN_1+W_1\right)$. With the help of \eqref{4.12} and \eqref{4.8} we find
\begin{align*}
\widetilde k(\widetilde s)=\vert \widetilde \nabla _{C^\prime }C^\prime \vert =\sqrt{\displaystyle\vert \widetilde g(\widetilde \nabla _{C^\prime }C^\prime ,\widetilde \nabla _{C^\prime }C^\prime )\vert }=\frac{\vert k_1(t)\vert }{a^2}=\frac{\vert \theta (\xi )(t)\vert }{2}.
\end{align*}
Thus we establish the truth of the first equality in \eqref{4.6} and \eqref{4.12} becomes
\begin{align}\label{4.13}
\widetilde \nabla _{C^\prime }C^\prime =\widetilde kE_2 .
\end{align}
Since $(C(t), {\bf F_1})$ is non-geodesic, from Corollary \ref{Corollary 3.4} it follows that $\theta (\xi )\neq 0$ along $C$. Hence $C(\widetilde s)$ is also non-geodesic.\\
Now, we compute
\begin{align}\label{4.14}
\widetilde \nabla _{C^\prime }E_2=\widetilde \nabla _{C^\prime }\widetilde \epsilon\left(\frac{1}{a}\varphi \dot C-\xi \right)=\frac{1}{\vert a\vert }\left(\frac{\widetilde \epsilon}{a}\widetilde \nabla _{\dot C}\varphi \dot C-\widetilde \epsilon \widetilde \nabla _{\dot C}\xi \right).
\end{align}
Further, by using \eqref{2.4} we get 
\begin{align}\label{4.15}
\widetilde \nabla _{\dot C}\varphi \dot C=\nabla _{\dot C}\varphi \dot C+
\frac{a^2\theta (\xi )}{2}\xi .
\end{align}
From the well known formula $(\nabla _{\dot C}\varphi )\dot C=\nabla _{\dot C}\varphi \dot C-
\varphi (\nabla _{\dot C}\dot C)$ we express
\begin{align}\label{4.16}
\nabla _{\dot C}\varphi \dot C=(\nabla _{\dot C}\varphi )\dot C+\varphi (\nabla _{\dot C}\dot C) .
\end{align}
By virtue of \eqref{2.2} we find
\begin{align*}
(\nabla _{\dot C}\varphi )\dot C=-\frac{a\theta (\xi )}{2}(a\xi +\varphi ^2\dot C) .
\end{align*}
Taking into account \eqref{4.8} and \eqref{4.9} we have
\begin{align*}
\varphi (\nabla _{\dot C}\dot C)=\varphi (k_1W_1)=\frac{a^2\theta (\xi )}{2}\varphi ^2\dot C .
\end{align*}
Substituting the latter two equalities in \eqref{4.16} we obtain 
\begin{align}\label{4.17}
\nabla _{\dot C}\varphi \dot C=-\frac{a^2\theta (\xi )}{2}\xi .
\end{align}
From \eqref{4.15} and \eqref{4.17} it follows
\begin{align}\label{4.18}
\widetilde \nabla _{\dot C}\varphi \dot C=0 .
\end{align}
By using \eqref{2.4} and \eqref{2.3} we get
\begin{align}\label{4.19}
\widetilde \nabla _{\dot C}\xi =\nabla _{\dot C}\xi =\frac{\theta (\xi )}{2}\varphi \dot C .
\end{align}
Substituting \eqref{4.18} and \eqref{4.19} in \eqref{4.14} we receive
$\widetilde \nabla _{C^\prime }E_2=-\widetilde k\frac{1}{\vert a\vert }\varphi \dot C$. We rewrite the last equality in the following equivalent form 
\begin{align*}
\widetilde \nabla _{C^\prime }E_2=-\widetilde kE_1+\widetilde kE_1-\widetilde k\frac{1}{\vert a\vert }\varphi \dot C=-\widetilde kE_1-\widetilde k\left(\frac{1}{\vert a\vert }\varphi \dot C-E_1\right)
\end{align*}
and put $E_3(\widetilde s)=\frac{1}{\vert a\vert }\varphi \dot C-E_1=\frac{1}{\vert a\vert }(\varphi \dot C-\dot C)=\epsilon \left(-\frac{1}{a}\dot C+W_1\right)$. Immediately we verify that $\widetilde g(E_3,E_3)=-1$, $\widetilde g(E_1,E_3)=\widetilde g(E_2,E_3)=0$. Now, we obtain
\begin{align}\label{4.20}
\widetilde \nabla _{C^\prime }E_2=-\widetilde kE_1-\widetilde \tau E_3 .
\end{align}
where $\widetilde \tau =\widetilde k$. Finally, we have
\begin{align*}
\widetilde \nabla _{C^\prime }E_3=\frac{1}{\vert a\vert }\widetilde \nabla _{\dot C}\frac{1}{\vert a\vert }(\varphi \dot C-\dot C)=\frac{1}{a^2}\left(\widetilde \nabla _{\dot C}\varphi \dot C-\widetilde \nabla _{\dot C}\dot C\right).
\end{align*}
Taking into account \eqref{4.13} and \eqref{4.18}, we infer 
\begin{align}\label{4.21}
\widetilde \nabla _{C^\prime }E_3=-\widetilde \tau E_2 .
\end{align}
The equalities \eqref{4.13}, \eqref{4.20} and \eqref{4.21} show that $C(\widetilde s)$ is a Frenet curve of osculating order 3. Note that in our case $\epsilon _1=\epsilon _2=-\epsilon _3=1$.\\
(ii) The truth of the assertion follows from Corollary \ref{Corollary 3.4} and \eqref{4.6}.
\end{proof}
As an immediate consequence from Proposition \ref{Proposition 3.7} and Theorem \ref{Theorem 4.2} we obtain
\begin{cor}
Let $(C(t), {\bf F_1})$ and $C(\widetilde s)$ be the curves from Theorem \ref{Theorem 4.2}. Then
\par
(i) $C(\widetilde s)$ is a generalized helix.
\par
(ii) If $(C(t), {\bf F_1})$ is a $\varphi $-geodesic, then $C(\widetilde s)$ is a proper helix.
\end{cor}
\section{Null $\varphi $-slant curves in a Lie group as a 3-dimensional $\F_4$-manifold and their matrix representation}
Let $G$ be a 3-dimensional real connected Lie group and $\mathfrak {g}$ be its Lie algebra with a basis 
$\{E_1, E_2, E_3\}$ of left invariant vector fields. We define an almost contact structure $(\varphi, \xi, \eta )$ and a left invariant B-metric $g$ as follows:
\begin{align*}\label{}
\begin{array}{llll}
\varphi E_1=E_2, \, \,  \varphi E_2=-E_1, \, \, \varphi E_3=0, \quad \xi =E_3, \, \, \eta (E_3)=1, 
\eta (E_1)=\eta (E_2)=0,\\
g(E_1,E_1)=-g(E_2,E_2)=g(E_3,E_3)=1, \quad
g(E_i,E_j)=0, \, i\neq j \in\{1,2,3\}.
\end{array}
\end{align*}
Let $(G,\varphi,\xi,\eta,g)$ be a 3-dimensional almost contact B-metric manifold  such that the Lie algebra $\mathfrak {g}$ of $G$ is determined by the following commutators:
\begin{align}\label{5.1}
[E_1,E_3]=\alpha E_2 \quad [E_2,E_3]=-\alpha E_1, \quad [E_1,E_2]=0, \,\, \alpha \in \R, \, \, \alpha \neq 0. 
\end{align}
Further we will show that if \eqref{5.1} holds, then $(G,\varphi,\xi,\eta,g)$ is an $\F_4$-manifold. By using the Koszul formula
\begin{align}\label{5.2}
2g(\nabla _{E_i}E_j,E_k)=g([E_i,E_j],E_k)+g([E_k,E_i],E_j)+g([E_k,E_j],E_i)
\end{align}
we obtain the following equality for the components $F_{ijk}=F(E_i,E_j,E_k)$,\\
$i,j,k \in \{1,2,3\}$ of the tensor $F$:
\begin{align}\label{5.3}
\begin{array}{ll}
2F_{ijk}=g([E_i,\varphi E_j]-\varphi [E_i,E_j],E_k)+g(\varphi [E_k,E_i]-[\varphi E_k,E_i],E_j)\\
\quad\quad\,\,\,+g([E_k,\varphi E_j]-[\varphi E_k,E_j],E_i).
\end{array}
\end{align}
By virtue of \eqref{5.1} and \eqref{5.3} we obtain that the non-zero components $F_{ijk}$ are
\begin{align}\label{5.4}
F_{113}=F_{131}=\alpha , \quad F_{223}=F_{232}=-\alpha .
\end{align}
For the tensor $F$ of a 3-dimensional $\F_4$-manifold, using \eqref{2.2}, we get 
\begin{align*}
F(X,Y,Z)=\frac{\theta (\xi )}{2}\{(X^1Y^1-X^2Y^2)Z^3+(X^1Z^1-X^2Z^2)Y^3\},
\end{align*}
where $X=X^iE_i$, $Y=Y^iE_i$, $Z=Z^iE_i$ are arbitrary vector fields. The latter equality and \eqref{5.4} imply that $(G,\varphi,\xi,\eta,g)$ is a 3-dimensional $\F_4$-manifold and $\alpha =\frac{\theta (\xi )}{2}$. With the help of \eqref{5.1} and \eqref{5.2} we find the components of the Levi-Civita connection $\nabla $. The non-zero ones of them are
\begin{align}\label{5.5}
\nabla _{E_1}E_2=\nabla _{E_2}E_1=\alpha \xi , \quad \nabla _{E_1}\xi =\alpha E_2, \quad 
\nabla _{E_2}\xi =-\alpha E_1.
\end{align}
Consider the curve $C(t)=e^{tX}$ on $G$, where
$t\in {\R}$ and $X\in \mathfrak {g}$. Hence the tangent vector to $C(t)$ at the identity element $e$ of 
$G$ is $\dot {C}(0)=X$. Let the coordinates $(p,q,r)$ of $\dot C$ with respect to the basis $\{E_1, E_2, E_3\}$ are given by
\begin{align}\label{5.6}
p=-\epsilon \sqrt{\frac{\sqrt{a^4+b^2}-a^2}{2}}, \quad q=\sqrt{\frac{\sqrt{a^4+b^2}+a^2}{2}},
\quad r=a,
\end{align}
where $a, b\in \R$, $(a,b)\neq (0,0)$ and $\epsilon =\{{\rm sign}\, b\}=\{\pm 1\}$. It is easy to see that
$g(\dot {C},\dot {C})=0$ and $\eta (\dot C)=a$. Also, having in mind that $\varphi \dot C=(-q,p,0)$,
we have $g(\dot C,\varphi \dot C)=b$. Hence, $C(t)$ is a $\varphi $-slant null curve in $(G,\varphi,\xi,\eta,g)$. Furthermore, using \eqref{5.5}, one obtains
\begin{align}\label{5.7}
\begin{array}{lll}
\nabla _{\dot C}{\dot C}=\alpha (-aqE_1+apE_2-b\xi )\\
\qquad \, \, \, =\displaystyle\alpha \sqrt{a^4+b^2}\left(\frac{-aq}{\sqrt{a^4+b^2}}E_1+\frac{ap}{\sqrt{a^4+b^2}}E_2-\frac{b}{\sqrt{a^4+b^2}}E_3\right)\\
\qquad \, \, \, =\alpha \sqrt{a^4+b^2}W_1,
\end{array}
\end{align}
where the vector field
\begin{align}\label{W_1}
W_1=\left(\frac{-aq}{\sqrt{a^4+b^2}},\frac{ap}{\sqrt{a^4+b^2}},-\frac{b}{\sqrt{a^4+b^2}}\right)
\end{align}
is a spacelike unit. 
Then the unique $N_1$ corresponding to $W_1$  is given by
\begin{align}\label{N_1}
N_1=\left(-\frac{a^2p+2bq}{2(a^4+b^2)},\frac{-a^2q+2bp}{2(a^4+b^2)},\frac{a^3}{2(a^4+b^2)}\right).
\end{align}
Thus, by using \eqref{5.5}, we obtain
\begin{align}\label{5.8}
\begin{array}{ll}
\nabla_{\dot C}N_1=\displaystyle\frac{\alpha a^2}{2\sqrt{a^4+b^2}}W_1, \\
\nabla_{\dot C}W_1=-\displaystyle\frac{\alpha a^2}{2\sqrt{a^4+b^2}}\dot C-
\alpha \sqrt{a^4+b^2}N_1 .
\end{array}
\end{align}
Comparing the equations \eqref{5.7} and \eqref{5.8} with \eqref{general Frenet eq} we get
\begin{align*}
h=0 , \qquad k_1=\alpha \sqrt{a^4+b^2}, \qquad k_2=\displaystyle\frac{\alpha a^2}
{2\sqrt{a^4+b^2}} 
\end{align*}
with respect to to the Frenet frame ${\bf F_1}=\{\dot C, N_1, W_1\}$.\\
Further, we find the matrix representation of $C(t)$ and ${\bf F}_1$. 
Let us recall that the adjoint representation $\rm Ad$ of $G$ is the following Lie group homomorphism
\[
\rm {Ad} : G \longrightarrow Aut({\g}).
\]
For $X\in {\g}$, the map ${\rm {ad}}_X : {\g}\longrightarrow {\g}$ is defined by ${\rm {ad}}_X(Y)=[X,Y]$, where by ${\rm ad}_X$ is denoted  ${\rm {ad}}(X)$. Due to the Jacobi identity, the map
\[
\rm {ad} : {\g} \longrightarrow End({\g}) : X\longrightarrow ad_X
\]
is Lie algebra homomorphism, which is called  adjoint representation of ${\g}$.
Since the set ${\rm End}({\g})$ of all ${\K}$-linear maps from ${\g}$ to ${\g}$ is isomorphic to the set of all  $(n\times n)$ matrices ${\rm M}(n,{\K})$ with entries in ${\K}$, $\rm {ad}$ is a matrix representation of ${\g}$. We denote by $M_i$ the matrices of ${\rm ad}_{E_i}$ (i=1,2,3) with respect to the basis $\{E_1,E_2,E_3\}$ of ${\g}$.  Then for an arbitrary $X=x_1E_1+x_2E_2+x_3E_3$ ($x_1, x_2, x_3 \in {\R}$) in ${\g}$ the matrix $A$ of ${\rm ad}_X$ is $A=x_1M_1+x_2M_2+x_3M_3$.
Then by virtue of the well known identity $e^A={\rm {Ad}}\left(e^X\right)$ we find the matrix representation of the Lie group $G$. By using \eqref{5.1} we obtain $M_1$, $M_2$, $M_3$ and then $A$
\[
M_1=\left(\begin{array}{lll}
0 & 0 & 0 \cr
0 & 0 & \alpha \cr
0 & 0 & 0
\end{array}\right), \quad
M_2=\left(\begin{array}{lcr}
0 & 0 & -\alpha \cr
0 & 0 & 0 \cr
0 & 0 & 0
\end{array}\right), \quad
M_3=\left(\begin{array}{rll}
0 & \alpha & 0 \cr
-\alpha & 0  & 0 \cr
0 & 0 & 0
\end{array}\right),
\]
\begin{align}\label{5.9}
A=\left(\begin{array}{ccc}
0 & x_3\alpha & -x_2\alpha \cr
-x_3\alpha & 0  & x_1\alpha  \cr
0 & 0 & 0
\end{array}\right).
\end{align}
The characteristic polynomial of A is
\[
P_A(\lambda )=-\lambda(\lambda ^2+x_3^2\alpha ^2) =0 .
\]
Hence for the eigenvalues $\lambda _i \, (i = 1, 2, 3)$ of $A$ we have
\[
\lambda _1=0 , \quad \lambda _2=ix_3\alpha  , \quad \lambda _3=-ix_3\alpha , \quad i^2=-1.
\]
By the assumption that $x_3\neq 0$, the eigenvectors
\[
p_1=(x_1,x_2,x_3), \quad p_2=(1,i,0),\quad p_3=(i,1,0)
\]
corresponding to $\lambda _1, \lambda _2, \lambda _3$, respectively, are linearly independent for arbitrary $x_1$, $x_2$ and $x_3\neq 0$. For the change of basis matrix P and its inverse matrix 
$P^{-1}$ we get
\[
P=\left(\begin{array}{rll}
x_1 & 1 & i \cr
x_2 & i & 1 \cr
x_3 & 0 & 0
\end{array}\right) , \quad 
P^{-1}=\frac{1}{2x_3}\left(\begin{array}{ccc}
0 & 0 & 2 \cr
x_3 & -ix_3 & -x_1+ix_2 \cr
-ix_3 & x_3 & -x_2+ix_1
\end{array}\right) .
\]
By using that $e^A = Pe^JP^{-1}$, where $J$ is the diagonal matrix with elements 
$J_{ii} =\lambda _i$, we obtain the matrix representation of the Lie group $G$ in case  $x_3\neq 0$
\begin{align}\label{5.10}
\small G=\left\{e^A=
\left(\begin{array}{ccc}
\cos \alpha x_3 & \sin \alpha x_3  & \frac{x_1}{x_3}(1-\cos \alpha x_3)-\frac{x_2}{x_3}\sin \alpha x_3  \cr \cr
-\sin \alpha x_3 & \cos \alpha x_3  & \frac{x_2}{x_3}(1-\cos \alpha x_3)+\frac{x_1}{x_3}\sin \alpha x_3 \cr \cr
0 & 0 & 1
\end{array}\right)
\right\} .
\end{align}
The coordinates of the vector field $t\dot C\in {\g}$, $t\in \R$, are $(tp,tq,ta)$, where $p,q$ are given by \eqref{5.6} and $a\neq 0$. Since ${\rm Ad}(C(t))={\rm Ad}\left(e^{t\dot c}\right)$,
we find ${\rm Ad}(C(t))$ replacing $x_1$, $x_2$ and $x_3$ in \eqref{5.10} with
$tp$, $tq$ and $ta$, respectively. Thus, for the matrix representation of a $\varphi $-slant null curve 
$C(t)$, which is not a Legendre curve, we have
\begin{align*}
{\rm Ad}(C(t))=
\left(\begin{array}{ccc}
\cos \alpha at & \sin \alpha at  & \frac{p}{a}(1-\cos \alpha at)-\frac{q}{a}\sin \alpha at  \cr \cr
-\sin \alpha at & \cos \alpha at  & \frac{q}{a}(1-\cos \alpha at)+\frac{p}{a}\sin \alpha at \cr \cr
0 & 0 & 1
\end{array}\right) .
\end{align*}
Finally, we may obtain the  matrix representations of $\dot C$, $W_1$ and $N_1$ replacing $x_1$, $x_2$ and $x_3$ in \eqref{5.9} with their coordinates, determined by   \eqref{5.6},  \eqref{W_1} and  \eqref{N_1}, respectively.

\end{document}